\documentclass[12pt]{amsart}
\usepackage{bbm}
\usepackage{mathabx}
\usepackage{tensor}
\usepackage{enumerate}
\usepackage{amssymb}
\usepackage{mathtools}
\usepackage{amsmath}
\usepackage{amsfonts}
\usepackage{amsthm}
\usepackage{stmaryrd}
\usepackage{paralist}
\usepackage{indentfirst}
\usepackage[all]{xy}
\usepackage{mathrsfs}
\usepackage{graphicx}
\usepackage{hyperref}
\usepackage{color}
\usepackage{tikz-cd}
\usepackage{tkz-euclide}
\usepackage{tikz}
\usetikzlibrary{matrix}
\usepackage{enumitem}
\usepackage{amsmath,amscd}
\usepackage{pifont}
\usetikzlibrary{arrows}
\usepackage[all]{xy}
\usepackage{graphicx}
\usepackage{placeins}
\usepackage{xspace}

\usetikzlibrary{calc,intersections,through,backgrounds}

\def\adic{\textrm{-adic}}

\def\GQ{\Gal_\QQ}
\def\GQp{\Gal_{\QQ_p}}

\def\ve{\varepsilon}
\def\rank{\mathrm{rank}}

\newtheorem*{theorem*}{Theorem~\ref{thm2}}
\newtheorem{thm}{Theorem}[section]
\newtheorem{theorem}[thm]{Theorem}

\newtheorem{lemma}[thm]{Lemma}
\newtheorem{corollary}[thm]{Corollary}

\newtheorem{conjecture}[thm]{Conjecture}

\newtheorem{proposition}[thm]{Proposition}

\theoremstyle{definition}

\newtheorem{definition}[thm]{Definition}
\newtheorem{notation-definition}[thm]{Notation-Definition}

\newtheorem{statement}[thm]{Statement}
\newtheorem{remark}[thm]{Remark}

\newtheorem{notation}[thm]{Notation}

\numberwithin{equation}{section}

\def\olr{\bar{\rho}}

\def\tk{\tilde{k}}

\def\omin{\mathrm{min}}

\def\calK{\mathcal{K}}

\def\calO{\mathcal{O}}

\def\calS{\mathcal{S}}

\def\CC{\mathbb{C}}
\def\FF{\mathbb{F}}

\def\QQ{\mathbb{Q}}
\def\RR{\mathbb{R}}

\def\ZZ{\mathbb{Z}}

\def\bfm{\mathbf{m}}

\def\bT{\mathbf{T}}

\def\rmK{\mathrm{K}}

\def\tH{\widetilde{\mathrm{H}}}

\def\Or{\bar{r}}

\def\OD{{\Delta}}

\def\kob{{k_{0\bullet}^\omin}}
\def\ko{{k_0^\omin}}
\def\kO{k_0^{\mathrm{max}}}
\def\NP{\mathrm{NP}}
\def\Dig{\mathrm{\Dig}}
\def\v{v_p}

\def\nS{\mathrm{nS}}

\DeclareMathOperator{\GL}{GL}

\DeclareMathOperator{\Hom}{Hom}

\DeclareMathOperator{\Gal}{Gal}

\newcommand{\ur}{\mathrm{ur}}

\newcommand{\Iw}{\mathrm{Iw}}

\newcommand{\new}{\mathrm{new}}

\begin{document}
	\title{The localized Gouvêa--Mazur conjecture}

	\author{Rufei Ren}
	\address{Department of Mathematics 
		Fudan University}
	\email{rufeir@fudan.edu.cn}
	\date{\today}
	\begin{abstract}
%
Gouvêa--Mazur \cite{gm} made a conjecture on the local constancy of slopes of modular forms when the weight varies $p$-adically. Since one may decompose the space of modular forms according to associated residual Galois representations, the Gouvêa--Mazur conjecture makes sense for each such component.  We prove the localized Gouvêa--Mazur conjecture when the residual Galois representation is irreducible and its restriction to $\Gal(\bar\QQ_p/\QQ_p)$ is reducible and  generic.
	\end{abstract}

	\subjclass[2010]{}
	\maketitle
	
	\setcounter{tocdepth}{1}
	\tableofcontents

\section{Introduction}

Let $p$ be a prime number, $N$ be a positive integer coprime to $p$. We normalize the $p$-adic valuation of elements in $\bar\QQ_p$ so that $v_{p}(p)=1$.
For any $r\in \QQ$ and a space $S$ of modular forms with $U_p$-action, we write $d(S, r)$ for the cardinality of the number of generalized eigenvalues of the $U_p$-action with $p$-adic valuation $r$. 
Let $S_{k}(\Gamma_0(Np))$ be the space of cusp forms with weight $k$ and level $Np$.
\begin{statement}\label{state:1}
	There is a polynomial $M(x)\in \QQ[x]$ such that  
	for any $r\in \QQ_{\geq 0}$, if $k_{1}$ and $k_{2}$ are integers such that
\begin{enumerate}
	\item 	$k_{1}\geq 2r+2$ and $k_{2}\geq 2r+2$;
	\item  $k_{1} \equiv k_{2} \pmod {(p-1)p^{m}}$ for an integer $m\geq M(r)$,
\end{enumerate}
	then $d\left(S_{k_1}(\Gamma_0(Np)), r\right)=d\left(S_{k_2}(\Gamma_0(Np)), r\right)$.
\end{statement}

Based on the numerical data from Mestre, Gouvêa and Mazur made the following conjecture in \cite{gm}.
\begin{conjecture}[Gouvêa--Mazur Conjecture]\label{conj}
 Statement~\ref{state:1} is true with $M(x)=x$.
\end{conjecture} 
In \cite{Buz}, Buzzard and Calegari disproved this conjecture by giving an explicit counterexample.
However, it is still believed to be true that $M$ can be taken to be a
linear polynomial. So far, the only definite result 
concerning Statement~\ref{state:1} is due to \cite{wan}, in which Wan proved Statement~\ref{state:1} with a quadratic $M$. Recently, Hattori proved in \cite{Hat} an analogue of Gouvêa--Mazur conjecture for 
Drinfeld cuspforms of level $\Gamma_1(t)$. 

 In this paper, we prove that if we localize $S_{k}(\Gamma_0(Np))$ at some maximal
Hecke ideal associated to a Galois representation	$\bar r:\GQ:=\Gal\left(\bar{\QQ} / \QQ\right) \rightarrow \GL_{2}(\bar\FF_p)$ (under certain conditions on $\bar r$), Statement~\ref{state:1} is true  with $M(x)=x+5$ (see Theorem~\ref{thm1}). Our work
 is mainly based on \cite{xiao} and \cite{xiao1} from Liu, Truong, Xiao and Zhao, in which they solved the so-called local ghost conjecture from \cite{BP}. To present our main theorem, we introduce some notations first.

Let $E$ be a finite field extension of $\QQ_p$ that contains $\sqrt p$, and denote by  $\calO$ and $\FF$ its
ring of integers and residue field, respectively. Let $\Delta :=(\mathbb{Z} / p \mathbb{Z})^{\times}$.
For any  character $\psi:\Delta \to \bar\QQ^\times$ and any  Galois representation	$\bar r:\GQ \rightarrow \GL_{2}(\FF)$, we denote by
$S_{k}\left(\Gamma_0(Np),\psi\right)_{\mathfrak{m}_{\bar r}}$ the localization of $S_{k}\left(\Gamma_0(Np),\psi\right)$ at the maximal Hecke ideal $\mathfrak{m}_{\bar r}$ associated to $\bar r$. 
Since $S_{k}\left(\Gamma_0(Np),\psi\right)$ has finite dimension, it can be decomposed  into a direct sum $\bigoplus_{\bar r} S_{k}(\Gamma_0(Np), \psi)_{\mathfrak{m}_{\bar r}}$ over a finite set of (semisimple) residual Galois representations $\bar r:\GQ \rightarrow \GL_{2}\left(\FF\right)$. 
Therefore,  it makes sense to study a similar question as Conjecture~\ref{conj} about each $S_{k}\left(\Gamma_0(Np), \psi\right)_{\mathfrak{m}_{\bar r}}$. In  consideration of the following two remarks, we prove this result in Theorem~\ref{thm1} under some certain assumptions. \begin{enumerate}
	\item The proof of our main theorem relies heavily on the local ghost conjecture, whose key requirements are that $\bar{r}$ is absolutely irreducible, and $	\Or|_{\mathrm{I}_{\QQ_p}} $ is reducible and generic, i.e. there is a unique pair of integers $(a, b)$ such that 
	\begin{equation}\label{eq:1}
		\Or|_{\mathrm{I}_{\QQ_p}} \simeq\begin{bmatrix}
			\omega_{1}^{a+b+1} & * \\
			0 & \omega_{1}^{b}
		\end{bmatrix} \quad \text { for some~} 2 \leq a \leq p-5 \text { and } 0 \leq b \leq p-2,	
	\end{equation}
	where  $I_{\QQ_p}$ and 	$\omega_{1}$ are the inertial group and the first fundamental character of $\GQp$, respectively. 
	\item For $\bar r$ in (1), we have $$S_{k}\left(\Gamma_0(Np),\psi\right)_{\mathfrak{m}_{\bar r}}\neq \emptyset\ \textrm{if and only if}\ \psi(\bar\alpha)=\omega(\bar\alpha)^{a+2b-k+2}\ \textrm{for every}\ \bar\alpha\in \Delta.$$
\end{enumerate}
\begin{theorem}[Main theorem]\label{thm1}
Assume that $p\geq 11$. 
Let   $\bar r: \GQ \rightarrow \GL_{2}(\mathbb{F})$ be an absolutely irreducible residual Galois representation  such that 
	$	\Or|_{\mathrm{I}_{\QQ_p}} $
	is reducible and generic, i.e. $2 \leq a \leq p-5$; and $\psi:(\ZZ/p\ZZ)^\times \to \bar\QQ^\times$ be  a character.

	For any integers $m\geq 4$ and $k_1, k_2$ such that 
%
\begin{enumerate}
	\item $k_{1}> m-3$ and $k_2> m-3$,
	\item  $k_{1} \equiv k_{2} \pmod {(p-1)p^{m}}$, 
		\item 	 $\psi(\bar\alpha)=\omega(\bar\alpha)^{a+2b-k_1+2}=\omega(\bar\alpha)^{a+2b-k_2+2}$ for every $\bar\alpha\in \Delta$,
\end{enumerate}
 we have $d\left(S_{k_{1}}(\Gamma_0(Np),\psi)_{\mathfrak{m}_{\bar r}}, m-4\right)=d\left(S_{k_{2}}(\Gamma_0(Np),\psi)_{\mathfrak{m}_{\bar r}}, m-4\right)$.
\end{theorem} 
\begin{remark}
\begin{enumerate}
	\item 	The conditions $p\geq 11$ and $2\leq a\leq p-5$ are from one of the key ingredients, local ghost conjecture (see \cite[Theorem~1.3]{xiao1}), to the proof of Theorem~\ref{thm1}.
	\item The integer $m$ in Theorem~\ref{thm1} plays the same role to the rational number $r$ in Statement~\ref{state:1}.
\end{enumerate}
\end{remark}
We note that Theorem~\ref{thm1} is a direct consequence of the following theorem, which is our main technical result. See details in \S\ref{section:2}.
\begin{theorem*}
Assume that $p\geq 11,$ $2\leq a\leq p-5$, $0\leq b\leq p-2$, and that $\ve=\ve_{1} \times \ve_{1} \omega^{k_{\ve}-2}$ is $\bar\rho$-relevant.
For  any integers $m\geq 4$, $k_{1}\geq 2$ and $k_{2}\geq 2$ such that $k_{1} \equiv k_{2}\equiv k_\ve\pmod{p-1}$ and
$k_{1} \equiv k_{2}\pmod{p^{m}}$, the multiset of $U_p$-slopes  with slope value $\leq m-4$ of $G_{\olr}^{(\ve)}(w_{k_1},-)$ is same to the one of  $G_{\olr}^{(\ve)}(w_{k_2},-)$.
\end{theorem*}
It is important to clarify that our paper does not conclusively prove the Gouvêa--Mazur conjecture, instead indicating a discrepancy of an additive factor of $-4$. In contrast, the example offered by Buzzard and Calegari shows a deviation of an additive factor of $-1$. By synthesizing these two findings, we have managed to reduce the accurate factor in our studied case to the interval $[-4, -1]$. We trust that with a more rigorous estimation for \eqref{eqw3} or Lemma~\ref{lemma3}, we can confine the factor even further within a smaller interval.

	\subsection*{Acknowledgment}
This manuscript would not have been conceivable without the insightful contributions derived from the works of Ruochuan Liu, Nha Truong, Liang Xiao, and Bin Zhao \cite{xiao, xiao1}. We extend our profound gratitude to Liang Xiao in particular, for numerous enlightening discussions concerning the ghost conjecture. Furthermore, we wish to express our sincere appreciation to the anonymous referees for their invaluable recommendations on enhancing the structure and readability of our document. This work is partially supported by the Chinese NSF grant KRH1411532, two grants KBH1411247 and KBH1411268 from Shanghai Science and Technology Development Funds, and a grant from the New Cornerstone Science Foundation.

\section{The local ghost conjecture}\label{section:2}
%
In this section, we introduce the so-called local ghost series, and prove 
Theorem~\ref{thm1} by combining  \cite[Theorem~8.7(1)]{xiao1} (the local ghost conjecture) with Theorem~\ref{thm2}. 

 Let $E$ be a finite extension of $\QQ_p$ that contains $\sqrt{p}$, whose  
ring of integers, uniformizer and  residue field are denoted by $\calO$, $\varpi$ and $\FF$, respectively. Let  $\omega:\FF_p^\times\to \calO^\times$ be the Teichmüller lift, and $\omega_1: \mathrm{I}_{\mathbb{Q}_p} \rightarrow \operatorname{Gal}\left(\mathbb{Q}_p\left(\mu_p\right) / \mathbb{Q}_p\right) \cong \mathbb{F}_p^{\times}$ denote the first fundamental character of the inertial $\operatorname{subgroup} \mathrm{I}_{\mathbb{Q}_p}$ at $p$. Let  $\Delta:=(\ZZ/p\ZZ)^\times$,  $\bfm_{\CC_p}$ be the maximal ideal of $\calO_{\CC_p}$, and  $w_k:=\exp (p(k-2))-1$ for each $k \in \mathbb{Z}$.

\begin{definition}\label{re:2}
	Let $\bar r:\Gal_\QQ\to \GL_2(\FF)$ be an absolutely irreducible representation  such that 	$\bar\rho:=\bar r|_ {\mathrm{I}_{\mathbb{Q}_{p}} }: \mathrm{I}_{\mathbb{Q}_{p}} \rightarrow \mathrm{GL}_{2}(\mathbb{F})$ is reducible, i.e. 
		$$\bar\rho\simeq \begin{bmatrix}\omega_{1}^{a+b+1} & *  \\ 0 & \omega_{1}^{b}\end{bmatrix}$$
 for some  $a , b\in\{0, \ldots, p-2\}$.
	\begin{enumerate}
\item 	
		We call $\bar\rho$ \emph{generic} if  $a \in\{2, \ldots, p-5\}$,  \emph{split} if $*=0$ and \emph{nonsplit} if $* \neq 0$. 
		\item
For a character $\ve: \Delta^{2} \rightarrow \mathcal{O}^{\times}$, we define 
	$\ve_1:\Delta\rightarrow \mathcal{O}^{\times}$ and $k_\varepsilon\in \{2,\dots,p\}$ by 	\begin{equation*}
		\ve=\ve_{1} \times \ve_{1} \omega^{k_{\ve}-2}.
	\end{equation*}
A character $\ve$ is called \emph{$\bar\rho$-relevant} if    $\ve(\bar{\alpha}, \bar{\alpha})=\omega(\bar{\alpha})^{a+2b}$ for all $\bar{\alpha} \in \Delta$, in which case, we define  $s_\ve\in \{0,\dots,p-2\}$
	from $\ve=\omega^{-s_{\varepsilon}+b} \times \omega^{a+s_{\varepsilon}+b}.$
\item For any $\bar\rho$-relevant character $\ve$, in \cite[Definition~2.25]{xiao}, 
so-called \emph{the ghost series} 
$$G_{\olr}^{(\ve)}(w, t):=\sum_{n \geq 0} g_n^{(\ve)}(w) t^n \in \mathbb{Z}_{p}[w] \llbracket t \rrbracket$$
is constructed, where
\begin{itemize}
	\item $g_n^{(\ve)}(w)=\prod\limits_{k \equiv k_{\ve} \pmod {p-1}}\left(w-w_{k}\right)^{m_n^{(\ve)}(k)} \in \mathbb{Z}_{p}[w],
	$ and
	\item
	the exponents 
	$${\small
		m_n^{(\ve)}(k):=\begin{cases}
			\begin{aligned}
				&\min \Big\{n-d_{k}^{\ur}\left(\ve_{1}\right),d_{k}^{\Iw}\left(\widetilde\ve_{1} \right)-d_{k}^{\ur}\left(\ve_{1}\right)-n\Big\}
			\end{aligned}& \textrm{if}\ d_{k}^{\ur}\left(\ve_{1}\right)<n<d_{k}^{\Iw}\left(\widetilde\ve_{1} \right)-d_{k}^{\ur}\left(\ve_{1}\right),\\ 0& \textrm{otherwise},
	\end{cases}}$$
	where $\widetilde\ve_1:=\ve_1\times \ve_1$ and $d_{k}^{\ur}\left(\ve_{1}\right), d_{k}^{\Iw}(\widetilde\ve_{1} )$ are defined in \cite[Notation 2.18]{xiao}.
\end{itemize}

\item For a power series $f(t)=\sum_{n=0}^\infty a_nt^n\in \CC_p[\![t]\!]$, we call the lower convex hull of the points $\{\left(n, v_{p}\left(a_n\right)\right)\}_{n\geq 0}$ the \emph{Newton polygon of $f$},  denote it by $\mathrm{NP}(f)$.
	\end{enumerate}

\end{definition}

\begin{notation}
	\begin{enumerate}
		\item Let  $\rmK_p:=\GL_2(\ZZ_p)$. 
		\item A primitive $\mathcal{O} \llbracket \rmK_p \rrbracket$-projective augmented module of type $\bar{\rho}$ is a finitely generated projective (right) $\mathcal{O} \llbracket \mathrm{GL}_2\left(\mathbb{Z}_p\right) \rrbracket$-module $\widetilde{\mathrm{H}}$ with a (right) $\mathrm{GL}_2\left(\mathbb{Q}_p\right) /\left({ }^p{ }_p\right)^{\mathbb{Z}}$-action extending the $\mathrm{GL}_2\left(\mathbb{Z}_p\right)$ action from the module structure, such that $\overline{\mathrm{H}}:=\widetilde{\mathrm{H}} /\left(\varpi, \mathrm{I}_{1+p \mathrm{M}_2\left(\mathbb{Z}_p\right)}\right)$ is isomorphic to the projective envelope of $\operatorname{Sym}^a \mathbb{F}^{\oplus 2} \otimes \operatorname{det}^b$ as an $\mathbb{F}\left[\mathrm{GL}_2\left(\mathbb{F}_p\right)\right]$-module, where $\mathrm{I}_{1+p \mathrm{M}_2\left(\mathbb{Z}_p\right)}$ is the augmentation ideal of the Iwasawa algebra $\mathcal{O} \llbracket 1+p \mathrm{M}_2\left(\mathbb{Z}_p\right) \rrbracket$.
		
		\item For a character $\ve: \Delta^{2} \rightarrow \mathcal{O}^{\times}$, we put
		$$
		\mathcal{O} \llbracket w \rrbracket^{(\varepsilon)}:=\mathcal{O} \llbracket \Delta \times \mathbb{Z}_p^{\times} \rrbracket \otimes_{\mathcal{O}\left[\Delta^2\right], \varepsilon} \mathcal{O} .
		$$
		
		Consider the universal character
		$$
		\begin{aligned}
			 \chi_{\text {univ }}^{(\varepsilon)}: B^{\mathrm{op}}\left(\mathbb{Z}_p\right):=\left(\begin{array}{cc}
				\mathbb{Z}_p^{\times} & 0 \\
				p \mathbb{Z}_p & \mathbb{Z}_p^{\times}
			\end{array}\right) &\longrightarrow\left(\mathcal{O} \llbracket w \rrbracket^{(\varepsilon)}\right)^{\times} \\
			\left(\begin{array}{cc}
				\alpha & 0 \\
				\gamma & \delta
			\end{array}\right)& \longmapsto \varepsilon(\bar{\alpha}, \bar{\delta}) \cdot(1+w)^{\log (\delta / \omega(\bar{\delta})) / p},
		\end{aligned}
		$$
 where $\omega: \mathbb{F}_p^{\times} \rightarrow \mathcal{O}^{\times}$  is the Teichmüller lift. The induced representation (for the right action convention)
		$$
		\begin{aligned}
			& \operatorname{Ind}_{B^{\text {op }}\left(\mathbb{Z}_p\right)}^{\mathrm{Iw}_p}\left(\chi_{\text {univ }}^{(\varepsilon)}\right):=\left\{\text { continuous functions } f: \operatorname{Iw}_p \rightarrow \mathcal{O} \llbracket w \rrbracket^{(\varepsilon)} ;\right. \\
			& \left.f(g b)=\chi_{\text {univ }}^{(\varepsilon)}(b) \cdot f(g) \text { for } b \in B^{\mathrm{op}}\left(\mathbb{Z}_p\right) \text { and } g \in \operatorname{Iw}_p\right\}.
		\end{aligned}
		$$
		
		\item  For an $\calO\llbracket \rmK_p\rrbracket$-projective augmented module $ \tH$ of type $\bar\rho$ and a $\bar\rho$-relevant character $\ve$, the space of abstract $p$-adic forms is defined as on \cite[Page 258]{xiao} by 
		$$
\mathrm{S}_{\widetilde{\mathrm{H}}, p \text {-adic }}^{(\varepsilon)}:=\operatorname{Hom}_{\mathcal{O}\left[\mathrm{Iw}_p\right]}\left(\widetilde{\mathrm{H}}, \operatorname{Ind}_{B^{\text {op }}\left(\mathbb{Z}_p\right)}^{\mathrm{Iw}_p}\left(\chi_{\text {univ }}^{(\varepsilon)}\right)\right) 
			 \cong \operatorname{Hom}_{\mathcal{O}\left[\mathrm{Iw}_p\right]}\left(\widetilde{\mathrm{H}}, \mathcal{C}^0\left(\mathbb{Z}_p ; \mathcal{O} \llbracket w \rrbracket^{(\varepsilon)}\right)\right),
		$$
		where   $\mathcal{C}^0\left(\mathbb{Z}_p ; \mathcal{O} \llbracket w \rrbracket^{(\ve)}\right)$ is the space of continuous functions on $\mathbb{Z}_p$ with values in $\mathcal{O} \llbracket w \rrbracket^{(\ve)}$.
		\item 
		We denote by $C_{\tH}^{(\ve)}(w,-)$ the characteristic power series of $U_p$ acting on 		$\mathrm{S}_{\tH, p\adic}^{(\ve)}.$
	\end{enumerate}

\end{notation}

We now prove Theorem~\ref{thm1} with assuming the following one, whose proof is given in \S\ref{s4}.

\begin{theorem}\label{thm2}
Assume that $p\geq 11,$ $2\leq a\leq p-5$, $0\leq b\leq p-2$, and that $\ve=\ve_{1} \times \ve_{1} \omega^{k_{\ve}-2}: (\ZZ/p\ZZ)^\times\times (\ZZ/p\ZZ)^\times\to \calO^\times$ is $\bar\rho$-relevant.
For  any integers $m\geq 4$, $k_{1}\geq 2$ and $k_{2}\geq 2$ such that $k_{1} \equiv k_{2}\equiv k_\ve\pmod{p-1}$ and
$k_{1} \equiv k_{2}\pmod{p^{m}}$, the multiset of $U_p$-slopes  with slope value $\leq m-4$ of $G_{\olr}^{(\ve)}(w_{k_1},-)$ is same to the one of  $G_{\olr}^{(\ve)}(w_{k_2},-)$.
\end{theorem}

\begin{proof}[Proof of Theorem~\ref{thm1} assuming Theorem~\ref{thm2}]
	We first point out that from our assumption (1) of this theorem and $k_i$ are integers, we have $k_i\geq 2$ for $i=1,2$.
	
Let $$K^{p}:=\left\{\begin{bmatrix}
	a&b\\
	c&d
\end{bmatrix} \in \GL_2(\hat \ZZ^{(p)})\;\Big|\;  N | c\right\} \subseteq \GL_{2}(\mathbb{A}_{f}^{p}).$$ We simply denote by $\widetilde{\mathrm{H}}(K^{p})_{\mathfrak{m}_{\bar r}}$ the $\bar r$-localized completed homology with tame level $K^{p}$; that is  
	$$	\widetilde{\mathrm{H}}\left(K^{p}\right)_{\mathfrak{m}_{\bar r}}:=
	\underset{n\to \infty}{\varprojlim}  \mathrm{H}_{1}^{\text {ét }}\left(Y(K^{p}(1+p^{n} \mathrm{M}_{2}(\mathbb{Z}_{p})))_{\bar{\mathbb{Q}}}, \mathcal{O}\right)_{\mathfrak{m}_{\bar r}}^{\text {cplx=1 }},
	$$
where for every $n\geq 1$,
\begin{itemize}
	\item  $Y(K^{p}(1+p^{n} \mathrm{M}_{2}(\mathbb{Z}_{p})))$ is the corresponding open modular curve for $K^{p}(1+p^{n} \mathrm{M}_{2}(\mathbb{Z}_{p}))$ over $\mathbb{Q}$, and 
		\item 
		$
		\mathrm{H}_{1}^{\text {ét }}\left(Y(K^{p}(1+p^{n} \mathrm{M}_{2}(\mathbb{Z}_{p})))_{\bar{\mathbb{Q}}}, \mathcal{O}\right)_{\mathfrak{m}_{\bar r}}^{\text {cplx=1 }}
		$
		denotes the subspace of the first étale homology of the modular curve $Y(K^{p}(1+p^{n} \mathrm{M}_{2}(\mathbb{Z}_{p})))$ localized at $\mathfrak{m}_{\bar r}$, on which a fixed complex conjugation acts by $1$. 
\end{itemize}
	As explained on page $249$ of \cite{xiao}, $\widetilde{\mathrm{H}}(K^{p})_{\mathfrak{m}_{\bar r}}$ is an $\calO\llbracket \rmK_p\rrbracket$-projective augmented module of type $\bar\rho$, and for both $i=1,2$, we have
	\begin{equation}\label{eq::1}
		S_{k_i}(\Gamma_0(Np),\psi)_{\mathfrak{m}_{\bar r}}\cong \Hom_{\mathcal{O} \llbracket \Iw_p\rrbracket}\left(\widetilde{\mathrm{H}}\left(K^{p}\right)_{\mathfrak{m}_{\bar r}}, \operatorname{Sym}^{k_i-2} \mathcal{O}^{\oplus 2}\otimes (1\times \psi)\right)
	\end{equation}
as Hecke modules.
From the assumption (1) of this theorem, for both $i=1,2$, we have 
\begin{equation}\label{eq::2}
	\Hom_{\mathcal{O} \llbracket \Iw_p \rrbracket}\left(\widetilde{\mathrm{H}}\left(K^{p}\right)_{\mathfrak{m}_{\bar r}}, \operatorname{Sym}^{k_i-2} \mathcal{O}^{\oplus 2}\otimes (1\times \psi)\right)\subseteqq
S_{\widetilde{\mathrm{H}}\left(K^{p}\right)_{\mathfrak{m}_{\bar r}},p\adic}^{(1\times \omega^{a+2b})}\otimes_{\mathcal{O} \llbracket w \rrbracket, w \mapsto w_{k_i} }\mathcal{O}
\end{equation} as 
	Hecke modules, and the space on left hand side contains the first $\rank_\calO(S_{k_i}(\Gamma_0(Np),\psi)_{\mathfrak{m}_{\bar r}})$-st $U_p$-slopes of the one on right hand side. This implies that the sets of $U_p$-slopes with value $<k_i-1$ of the two spaces in \eqref{eq::2}  are the same. Hence, by our assumption (1) of this theorem, it is enough to show that 
$$d\left(	S_{\widetilde{\mathrm{H}}\left(K^{p}\right)_{\mathfrak{m}_{\bar r}},p\adic}^{(1\times \omega^{a+2b})}\otimes_{\mathcal{O} \llbracket w \rrbracket, w \mapsto w_{k_1} }\mathcal{O}, m-4\right)=d\left(	S_{\widetilde{\mathrm{H}}\left(K^{p}\right)_{\mathfrak{m}_{\bar r}},p\adic}^{(1\times \omega^{a+2b})}\otimes_{\mathcal{O} \llbracket w \rrbracket, w \mapsto w_{k_2} }\mathcal{O}, m-4\right).$$

By \cite[Theorem~8.7]{xiao1},  for both $i=1,2$, 
	the Newton polygon $\NP\left(C_{\widetilde{\mathrm{H}}\left(K^{p}\right)_{\mathfrak{m}_{\bar r}}}^{(\ve)}(w_{k_i},-)\right)$ is same to  $\NP\left(G_{\olr}^{(\ve)}(w_{k_i},-)\right)$, stretched in both $x$- and $y$-directions by $m\left(\widetilde{\mathrm{H}}\left(K^{p}\right)_{\mathfrak{m}_{\bar r}}\right)$ (defined in \cite[Definition~8.2]{xiao1}).
Combined with Theorem~\ref{thm2}, this implies that 
the $U_p$-slopes  with the slope value $\leq m-4$ of $C_{\widetilde{\mathrm{H}}\left(K^{p}\right)_{\mathfrak{m}_{\bar r}}}^{(\ve)}(w_{k_1},-)$ is same as those of $C_{\widetilde{\mathrm{H}}\left(K^{p}\right)_{\mathfrak{m}_{\bar r}}}^{(\ve)}(w_{k_2},-)$.
Combining this with \eqref{eq::1} and \eqref{eq::2}, we complete the proof.
\end{proof}

	\section{Near-Steinberg ranges}\label{s3}
From now on,  we fix some prime $p\geq 11$, a pair of integers $(a,b)$ with $2\leq a\leq p-5$, $0\leq b\leq p-2$, and a $\bar\rho$-relevant character 
$\ve: \Delta^{2} \rightarrow \mathcal{O}^{\times}$. 
  For simplicity, we omit the subscript $\olr$ and superscript $(\ve)$ from the notation. By \cite[Remark~2.30(3)]{xiao}, we may and will assume $b=0$.
%
%
%
%
%
%
%
%

\begin{notation}\label{notation:2}
\begin{enumerate}
	\item Let $\calK:=\{k\geq 2\;|\;k\equiv k_\ve\pmod {p-1}\}$. For any $k\in \calK$, we put $k:=k_{\ve}+k_{\bullet}(p-1)$,  $d_{k}^{\ur}:=d_{k}^{\ur}\left(\ve_{1}\right)$ (see Definition~\ref{re:2}(2) for  $\ve_1$), $d_{k}^{\Iw}:=d_{k}^{\Iw}\left(\widetilde\ve_1\right)$ and
	$d_k^\new:=d_k^\Iw-2d_k^\ur$.
	\item		We fix some $m\geq 4$ and $k_1\in \calK$. Let $\calS:=\{k\in \calK\;|\; k\equiv k_1\pmod {(p-1)p^m}\}$, and denote by $k_0^{\min }$
	and $\kO$ the two smallest  elements in $\calS$, i.e. $\ko\in \calS$ such that 
	 $2 \leq k_0^{\min }<2+(p-1) p^m$ and $k_0^{\max }:=k_0^{\min }+(p-1) p^m$. 
	
	\item	For any $k\in \calK$ and $\ell \in \{-\frac{1}{2} d_{k}^\new,-\frac{1}{2} d_{k}^\new+1, \ldots, \frac{1}{2} d_{k}^\new\}$, we put 
		\begin{equation*}
		\Delta'_{k, \ell}:=v_{p}\left(g_{\frac{1}{2} d_{k}^{\mathrm{Iw}}+\ell, \hat{k}}\left(w_{k}\right)\right)-\frac{k-2}{2} \ell, 
	\end{equation*}
	where $g_{\frac{1}{2} d_{k}^{\mathrm{Iw}}+\ell, \hat{k}}\left(w\right)$ is the polynomial in $w$ constructed by removing the $(w-w_k)$-factors in $g_{\frac{1}{2} d_{k}^{\mathrm{Iw}}+\ell}\left(w\right)$. 
	By \cite[Proposition~4.18(4)]{xiao}, we have
	\begin{equation}\label{neq::19}
		\Delta'_{k, \ell}=\Delta'_{k,-\ell} \quad \text { for all } \ell=-\frac{1}{2} d_{k}^{\new}, -\frac{1}{2} d_{k}^{\new}+1, \ldots, \frac{1}{2} d_{k}^{\new}.
	\end{equation}
	
	\item	We denote by $\OD_{k}$ the 
	lower convex hull  of the set of points $$\left\{(\ell, \Delta'_{k, \ell})\;|\; \ell\in\{-\tfrac{1}{2}d_k^\new,\dots, \tfrac{1}{2} d_k^\new\}
	\right\},$$
	and by $\left(\ell, \OD_{k, \ell}\right)$ the corresponding points on this convex hull. 
	\item For any integer $n$, we denote by $\{n\}\in \{0, \ldots, p-2\}$ the residue of $n$ modulo $p-1$, and let $$\delta:=\tfrac{1}{p-1}(\{a+s\}+s-\{a+2s\})=0 \textrm{~or~} 1,$$
	where	$s=s_\ve$ is defined in Definition~\ref{re:2}(2). 
	
\item 	We define $t_1$, $t_2$ by  \begin{center}
		\begin{tabular}{ |c|c|c| } 
			\hline
			& $t_1 $  & $t_2$   \\ 
			\hline
			$a+s<p-1$	& 	$s+\delta$ &  	$a+s+\delta+2$ \\ 
			\hline
			$a+s\geq p-1$	& 	$\left\{a+s\right\}+\delta+1$& 	$s+\delta+1$\\ 
			\hline
		\end{tabular}
	\end{center}
and 
	$\beta_{n}:= \begin{cases}t_1 & \text { if } n \text { is even}, \\ t_2-\frac{p+1}{2} & \text { if } n \text { is odd}.\end{cases}$

	\item For any integer $n$, let $$\theta_n:=\beta_{n-1}-\beta_{n}+\frac{p+1}{2}\quad \textrm{and}\quad 	\eta_n:=\frac{p-1}{2} k_{\bullet}-\frac{p+1}{2} \delta+\beta_{n}-1.$$ 
	By  the table in \cite[Lemma~5.2]{xiao}, $\theta_n$ is
	equal to $a+2$ or $p-1-a$, and hence satisfies
	\begin{equation}\label{eqr10}
		\theta_n\leq p-3.
	\end{equation}
\end{enumerate} 
\end{notation}

\begin{lemma}\label{eqr11}
We have	
\begin{enumerate}
	\item 	$t_1+t_2=	s+\{a+s\}+2+2\delta$,
	\item $4\leq t_2-t_1 \leq p-3.$
\end{enumerate}

\end{lemma}

\begin{proof}
	(1) It follows from Notation~\ref{notation:2}(6) directly.
	
	(2) 	Note that by Notation~\ref{notation:2}(6), $t_2-t_1=a+2$ or $p-1-a$. Hence, our assumption $2\leq a\leq p-5$ from the beginning of this section forces $4\leq t_2-t_1\leq p-3$. 
\end{proof}

%
%

\begin{lemma}\label{lem:8}
	\begin{enumerate}
		\item For any $k\in \calK$, we have
		$$
			d_{k}^{\mathrm{Iw}}=2 k_{\bullet}+2-2 \delta
		\quad \textrm{and}\quad
		 d_{k}^{\mathrm{ur}}=\left\lfloor\frac{k_{\bullet}-t_{1}}{p+1}\right\rfloor+\left\lfloor\frac{k_{\bullet}-t_{2}}{p+1}\right\rfloor+2.$$
		Clearly, 
		\begin{equation}\label{eqw1}
			 \frac{2k_\bullet}{p+1}-2\leq d_{k}^\ur\leq \frac{2k_\bullet}{p+1}+2.
		\end{equation}
	\item	For any $k,k'\in \calK$ such that $k<k'$, we have
		$$d_{k}^\ur\leq d_{k'}^\ur\quad \textrm{and}\quad d_{k}^\Iw-d_{k}^\ur \leq d_{k'}^\Iw-d_{k'}^\ur,$$
		where the second inequality is strict if $k'_\bullet\geq k_\bullet+2$.
		\item 	
		For any distinct $k,k'\in \calK$, if $d_{k}^\ur=d_{k'}^\ur$ or $d^\Iw_{k}-d^\ur_{k}=d^\Iw_{k'}-d^\ur_{k'}$, then 
		$\v(k-k')=0$. 
	\end{enumerate}
\end{lemma}
\begin{proof}
The two equalities in  (1) are directly from \cite[Corollary~4.4 and Proposition~4.7]{xiao} and imply 
both the inequality in (1) and the first inequality in (2). The second inequality in (2) follows from 
\[d_{k'}^\Iw-d_{k}^\Iw-(d_{k'}^\ur-d_{k}^\ur)\geq 2(k'_\bullet-k_\bullet)-2\left\lceil \frac{k'_\bullet-k_\bullet}{p+1}\right\rceil\geq 0,\]
where the second inequality above is strict if $k'_\bullet-k_\bullet\geq 2.$

For (3), if $d^\Iw_{k}-d^\ur_{k}=d^\Iw_{k'}-d^\ur_{k'},$ combining (2) with our assumption $k\neq k'$, we have 
$k'_\bullet-k_\bullet=1$ and hence $\v(k'_\bullet-k_\bullet)=0$. 
Now we assume that $
	d_{k}^\ur=d_{k'}^\ur.
	$
Combing the second equality in (1), Lemma~\ref{eqr11}(2) and $d_{k}^\ur=d_{k'}^\ur$, we have $|k_{\bullet}-k'_\bullet|\leq p-3$, and hence 	$\v(k-k')=0$.
\end{proof}
\begin{corollary}\label{coro:1}
For any $k\in \calK$, if we write $k_\bullet=A+B(p+1)$ with $t_1 \leq A\leq t_1+p$ and $B\in \ZZ$, then 
\begin{enumerate}
	\item if $d_k^{\ur}$ is odd, then $A \leqslant t_2-1$,
	\item if $d_k^\ur$ is even, then $t_2 \leqslant A$. 
\end{enumerate}
\end{corollary}
\begin{proof}
	Due to the similarity, we only prove (1).
Combining Lemma~\ref{lem:8}(1) with our assumption $p\geq 11$ and that $d_k^{\ur}$ is odd in this case, we have $\left\lfloor\frac{A-t_{1}}{p+1}\right\rfloor+\left\lfloor\frac{A-t_{2}}{p+1}\right\rfloor$ is odd. 
Combined with $t_1\leq t_2-1$ from Lemma~\ref{eqr11}(2), this implies that $A \leqslant t_2-1$.
\end{proof}

\begin{notation-definition}[\cite{xiao}, Definition~5.11]\label{notation1}\noindent

(1) For every $k\in \calK$ and $w_\star\in\bfm_{\CC_p}$, we denote by $L_{w_{\star}, k}$ the largest number (if exists) in $\left\{1, \ldots, \frac{1}{2} d_{k}^{\new}\right\}$ such that
$$
v_{p}\left(w_{\star}-w_{k}\right) \geq \OD_{k, L_{w_{\star}, k}}-\OD_{k, L_{w_{\star}, k}-1}.
$$

(2) When such $L_{w_{\star}, k}$ exists, the open interval $\nS_{w_{\star}, k}:=\left(\frac{1}{2} d_{k}^{\Iw}-L_{w_{\star}, k}, \frac{1}{2} d_{k}^{\mathrm{Iw}}+L_{w_{\star}, k}\right)$ is called the \emph{near-Steinberg range} for the pair $\left(w_{\star}, k\right)$. We put $\overline{\nS}_{w_{\star}, k}:=\big[\frac{1}{2} d_{k}^{\Iw}-L_{w_{\star}, k}, \frac{1}{2} d_{k}^{\mathrm{Iw}}+L_{w_{\star}, k}\big]$ its closure.

(3) By \cite[Theorem~5.19(1)]{xiao}, for a fixed $w_\star\in \bfm_{\CC_p}$, the set of near-Steinberg ranges $\{\nS_{w_{\star}, k}\;|\;k\in\calK \}$ is nested; that is, for any two such open intervals, either they are disjoint or one is contained within the other. A near-Steinberg range $\nS_{w_{\star}, k}$ is said to be \emph{maximal} if it is not contained within any other ranges.
\end{notation-definition}

\begin{proposition}[\cite{xiao}, Proposition~5.16(1) and Theorem~5.19(2)]\label{xiao:pro}
	For any $w_\star\in \bfm_{\CC_p}$, we have the following.  
	\begin{enumerate}
		\item   For any distinct weights $k, k^{\prime}\in \calK$ with $v_{p}\left(w_{k^{\prime}}-w_{k}\right) \geq \OD_{k, L_{w_{\star}, k}}-\OD_{k, L_{w_{\star}, k}-1}$ the following exclusions must be observed:
		$$
		\tfrac{1}{2} d_{k^{\prime}}^{\mathrm{Iw}} \notin \overline{\mathrm{nS}}_{w_{\star}, k} \quad \text { and } \quad d_{k^{\prime}}^{\mathrm{ur}}, d_{k^{\prime}}^{\mathrm{Iw}}-d_{k^{\prime}}^{\mathrm{ur}} \notin \mathrm{nS}_{w_{\star}, k} .
		$$
		\item For any pair of integers $0\leq n'<n''$, the following statements are equivalent:
		\begin{enumerate}
			\item 
			$n'$ and $n^{\prime \prime}$ are $x$-coordinates of vertices at the end of a line segment in $\NP(G(w_{\star},-)).$
			\item There exists a weight $k\in \calK$ such that $\nS_{w_{\star},k}$ is maximal and equal to $(n',n'')$.
		\end{enumerate} 
	\end{enumerate}
\end{proposition}

\begin{lemma}\label{lem:1}
 For any $k,k'\in \calK$, if $L_{w_{k'},k}\geq 1$, then $\v(k-k')\geq 1$.
\end{lemma}
\begin{proof}
	By   \cite[Lemma~5.2]{xiao}, for any $k\in \calK$ and $\ell=1, \ldots, \frac{1}{2} d_{k}^{\new}$, we have
	$$
	\Delta'_{k, \ell}-\Delta'_{k, \ell-1}\geq \frac{3}{2}+\frac{p-1}{2}(\ell-1).
	$$
	Combined with \eqref{neq::19}, this system of inequalities  implies that for any $\ell=1, \ldots, \frac{1}{2} d_{k}^{\new},$ we have
$$	\frac{\Delta'_{k,\ell}-\Delta'_{k,0}}{\ell}=\frac{\Delta'_{k,-\ell}-\Delta'_{k,0}}{\ell}\geq \frac{3}{2},$$ and hence 
$	\OD_{k,1}-\OD_{k,0}\geq \frac{3}{2}$.
	From our assumption $L_{w_{k_1},k}\geq 1$, we have
	\begin{equation*}
			\v(k-k')+1=\v(w_k-w_{k'})\geq
	\OD_{k,1}-\OD_{k,0}\geq \frac{3}{2}.
	\end{equation*}
Note that 	$\v(k-k')$ is an integer, and hence has to be greater than or equal to $1.$
\end{proof}
%

\begin{lemma}\label{lemma3}
 For any $k\in \calK$, let $n=d_k^\ur$. Then
 	\begin{equation*}
 			\frac{k-2}{2}-\frac{(p-1)(\frac12d^\new_k-1)+\theta_n}{2}
 	\geq	\frac{p-1}{p+1}k_\bullet-	 \begin{cases}
 		2& \textrm{if~} n \textrm{~is odd},\\
 			1& \textrm{if~} n \textrm{~is even}.\\
 		\end{cases}
 	\end{equation*}
\end{lemma}
\begin{proof}
	By Lemma~\ref{lem:8}(1),  we have 
	\begin{equation*}
		\begin{split}
			&	\frac{k-2}{2}-\frac{(p-1)(\frac12d^\new_k-1)+\theta_n}{2}\\
			=&\frac{\{a+2s\}+k_\bullet(p-1)}{2}-\frac{p-1}{2}(k_\bullet-\delta-n)-\frac{\theta_n}{2}\\
			=&\frac{s+\{a+s\}}{2}+\frac{p-1}{2}\cdot n-\frac{\theta_n}{2}.
		\end{split}
	\end{equation*}
Write $k_\bullet=A+B(p+1)$ with $t_1 \leq A\leq t_1+p$ and $B\in \ZZ$.

When $n$ is odd, by Corollary~\ref{coro:1}(1), we have $t_1\leq A\leq t_2-1.$
Combined with Lemma~\ref{lem:8}(1), this chain of inequalities implies that 
$$n=\left\lfloor\frac{k_\bullet-t_{1}}{p+1}\right\rfloor+\left\lfloor\frac{k_\bullet-t_{2}}{p+1}\right\rfloor+2=2B+1=\frac{2(k_\bullet-A)}{p+1}+1\geq \frac{2(k_\bullet-t_2+1)}{p+1}+1.$$

By  Notation~\ref{notation:2}(7), we have
$$\theta_n=\beta_0-\beta_1+\frac{p+1}{2}=t_1-t_2+p+1,$$
and hence
\begin{align*}
	&	\frac{s+\{a+s\}}{2}+\frac{p-1}{2}\cdot n-\frac{\theta_n}{2}\\
	\geq& 
	\frac{s+\{a+s\}}{2}+\frac{p-1}{2}\left(\frac{2(k_\bullet-t_2+1)}{p+1}+1\right)-\frac{t_1-t_2+p+1}{2}\\
	=& \frac{p-1}{p+1}k_\bullet+\frac{2t_2}{p+1}-\frac{p+3}{p+1}- \delta\\
	\geq & \frac{p-1}{p+1}k_\bullet-2,
\end{align*}
where the last to second equality follows from Lemma~\ref{eqr11}.

When $n$ is even, by Corollary~\ref{coro:1}(2), we have
$ t_2\leq A\leq t_1+p.$
Combined with Lemma~\ref{lem:8}(1), this chain of inequalities implies that 
$$d_k^\ur=\left\lfloor\frac{k_\bullet-t_{1}}{p+1}\right\rfloor+\left\lfloor\frac{k_\bullet-t_{2}}{p+1}\right\rfloor+2=2B+2=\frac{2(k_\bullet-A)}{p+1}+2\geq \frac{2(k_\bullet-t_1-p)}{p+1}+2.$$

By  Notation~\ref{notation:2}(7), we have
$$\theta_n=\beta_1-\beta_0+\frac{p+1}{2}=t_2-t_1,$$
and hence
\begin{align*}
	&	\frac{s+\{a+s\}}{2}+\frac{p-1}{2}\cdot n-\frac{\theta_n}{2}\\
	\geq& 
	\frac{s+\{a+s\}}{2}+\frac{p-1}{2}\left(\frac{2(k_\bullet-t_1+1)}{p+1}\right)-\frac{t_2-t_1}{2}\\
	=& \frac{p-1}{p+1}k_\bullet+\frac{2(t_1-1)}{p+1}- \delta\\
	\geq & \frac{p-1}{p+1}k_\bullet-1,
\end{align*}
where the last to second equality follows from Lemma~\ref{eqr11}(1).
\end{proof}

\begin{proposition}\label{lemma2}
	For any $k\in \calK$ with $\frac12d_{k}^\new\geq 1$, we have $$\frac{k-2}{2}\geq \OD_{k,\frac12d_{k}^\new}-\OD_{k,\frac12d_{k}^\new-1}-4.$$
\end{proposition}
\begin{proof}
	We introduce temporary notations $D:=\frac12d_{k}^\new$ and $n:=d_k^\ur$.

	For the case $k_\bullet=0$, by Lemma~\ref{lem:8}(1), we have $d_k^\Iw\leq 2$, and hence
$	D\leq 1$. Combining with the assumption in this lemma, we must have 
$	\frac12d_{k}^\Iw=D=1$, $n=\delta=0$. 
By \eqref{neq::19}, we have
\begin{equation}\label{eqt:1}
		\Delta'_{k, 1}-\Delta'_{k, 0}=\Delta'_{k, -1}-\Delta'_{k, 0}=\frac{k-2}{2}-\sum_{k'\in \calK\backslash\{k\} \textrm{~s.t. } k'_\bullet<t_1} v_p(w_k-w_{k'})\leq \frac{k-2}{2}.
\end{equation}

By \cite[Lemma~5.2]{xiao}, we have
$$\OD_{k,1}-\OD_{k,0}=\OD'_{k,1}-\OD'_{k,0}. $$
Combined with \eqref{eqt:1}, this equality completes the proof of this case.

	Now we assume that $k_\bullet\geq 1.$
	We first note that it is actually showed in the proof of \cite[Lemma~5.5]{xiao}  that
			\begin{equation}\label{eqr2}
		\Delta_{k, D}^{\prime}-\Delta_{k, D-1}^{\prime} \leq \frac{(p-1)(D-1)+\theta_n}{2} +2+\gamma+\lfloor\log_p D \rfloor,
	\end{equation}
	where  $\gamma$ is the maximal $p$-adic valuation of the integers in $[\eta_n-\frac{p+1}{2}(D-1), \eta_n+\theta_n+\frac{p+1}{2}(D-1)]$.

By \cite[(5.2.5)]{xiao},  
we have  $ \eta_n-\frac{p+1}{2}(D-1)> 0$ and $\eta_n+\theta_n+\frac{p+1}{2}(D-1)\leq pk_\bullet$. These two inequalities together imply that
\begin{equation}\label{eqr1}
	\gamma\leq \log_p (pk_\bullet).
\end{equation} 

On the other hand, combining  $\OD_{k,D-1}'-\OD_{k,D-1}=0$ for $D=1$ with  \cite[Lemma~5.8]{xiao}, we have 
\begin{equation}\label{eqw3}
	\OD_{k,D-1}'-\OD_{k,D-1}\leq 3(\log_p D)^2.
\end{equation}

Together with $\OD_{k,D}=\OD'_{k,D}$, \eqref{eqr2} and \eqref{eqr1}, this inequality implies
 	\begin{equation}\label{eqr3}
 \Delta_{k, D}-\Delta_{k, D-1} \leq \frac{(p-1)(D-1)+\theta_n}{2}+2+ \log_p (pk_\bullet)+\lfloor\log_p D \rfloor+ 3(\log_p D)^2.
 \end{equation}
Combining \eqref{eqr3} with Lemma~\ref{lemma3}, we have
 \begin{align*}
 	&\frac{k-2}{2}-( \Delta_{k, D}-\Delta_{k, D-1})\\\geq&
 \frac{k-2}{2}-	\frac{(p-1)(D-1)+\theta_n}{2}-\left(2+ \log_p (pk_\bullet)+\lfloor\log_p D \rfloor+ 3(\log_p D)^2\right)\\
 	\geq &\frac{p-1}{p+1}k_{\bullet}-2-\left(2+ \log_p (pk_\bullet)+\lfloor\log_p D\rfloor+ 3(\log_p D)^2\right):=f(D,k_\bullet,p).
 \end{align*}
 Note that by Lemma~\ref{lem:8}(1), we have $$D=\frac{1}{2}d_k^\Iw-n\leq \frac{1}{2}d_k^\Iw\leq k_\bullet+1.$$
 For any fixed $k_\bullet\geq 1$ and $D$ such that $D\leq k_\bullet+1$, the function $f(D,k_\bullet,p)$ is a decreasing function of $p\geq 11$. Hence, 
  \begin{align*}
 \frac{k-2}{2}-( \Delta_{k, D}-\Delta_{k, D-1})\geq& \tfrac{5}{6}k_{\bullet}-5- \log_{11} (k_\bullet)-\lfloor\log_{11} D \rfloor- 3(\log_{11} D)^2\\
\geq & \tfrac{5}{6}k_{\bullet}-5- \log_{11} (k_\bullet)-\lfloor\log_{11} (k_\bullet+1) \rfloor- 3(\log_{11} (k_\bullet+1))^2.
 \end{align*}
Under the condition $k_\bullet\geq 1$,  a numerical calculation from Matlab (see Figure~\ref{aaa}) shows that 
\begin{equation}\label{eqt2} 
	\frac{k-2}{2}-( \Delta_{k, D}-\Delta_{k, D-1})\geq \begin{cases}
		-4.417 &\textrm{for~}k_\bullet=1, 2;\\
		-3.961 &\textrm{for~}k_\bullet\geq 3.
	\end{cases}
\end{equation}

\begin{figure}[h!]
	\centering
		\includegraphics[width=\linewidth]{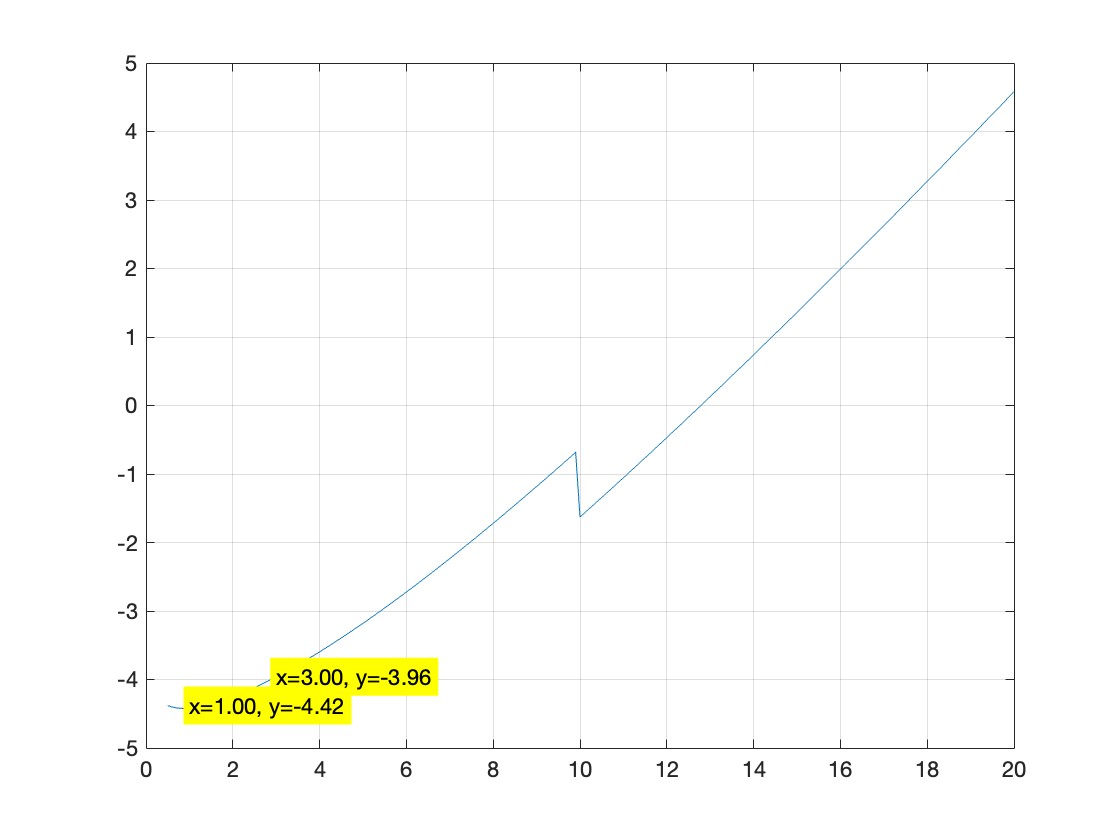}
		\caption{ The graph for $y=\tfrac{5}{6}x-5- \log_{11} (x)-\lfloor\log_{11} (x+1) \rfloor- 3(\log_{11} (x+1))^2$}
		\label{aaa}
	\end{figure}

Note that for $k_\bullet=1$ or $2$, by \cite[Lemmas~5.2 and 5.6]{xiao}, we have  $$\Delta_{k, D}-\Delta_{k, D-1}=\Delta'_{k, D}-\Delta'_{k, D-1}, $$
and hence $\frac{k-2}{2}-( \Delta_{k, D}-\Delta_{k, D-1})\in \frac{1}{2}\ZZ.$ Combining this with \eqref{eqt2}, we complete the proof.
\end{proof}

\begin{lemma}\label{lem:2}
	For any $k\in \calK$ such that $d_{k}^\new\geq 1$ and $w_\star\in \bfm_{\CC_p}$, if 
	$$\v(w_k-w_\star)\geq \OD_{k,\frac12d_{k}^\new}-\OD_{k,\frac12d_{k}^\new-1},$$
	then \begin{enumerate}
		\item $\nS_{w_\star,k}= (d^\ur_{k}, d^\Iw_{k}-d^\ur_{k})$.
		\item For any $k'(\neq k)\in \calK$ if  $\nS_{w_\star,k}\subseteq (d^\ur_{k'}, d^\Iw_{k'}-d^\ur_{k'})$, then
		$\v(k-k')=0$.
		\item The interval $\nS_{w_\star,k}$ is a maximal near-Steinberg range.
		\item 	$d^\ur_k$ and $d^\Iw_k-d^\ur_k$ are $x$-coordinates of vertices at the end of a line segment with slope $\frac{k-2}{2}$ in $\NP(G(w_{\star},-)).$
	\end{enumerate}
\end{lemma}
\begin{proof}
	(1) It follows directly from the definition of near-Steinberg range. 
	
	(2)
	From (1) and our assumption $\nS_{w_\star,k}\subseteq (d^\ur_{k'}, d^\Iw_{k'}-d^\ur_{k'})$, we have \begin{equation}\label{eq:2}
		d^\ur_{k}\geq d^\ur_{k'}\quad\textrm{and}\quad 
		d^\Iw_k-d^\ur_{k}\leq d^\Iw_{k'}-d^\ur_{k'}.
	\end{equation} 
	The first inequality above combined with the first one of  Lemma~\ref{lem:8}(2) and our assumption $k'\neq k$ implies $k' <k$. Together with the  second inequality of  Lemma~\ref{lem:8}(2), this implies that 	$$d^\Iw_k-d^\ur_{k}\geq d^\Iw_{k'}-d^\ur_{k'}.$$
	Combining it with \eqref{eq:2}, we have $$d^\Iw_k-d^\ur_{k}= d^\Iw_{k'}-d^\ur_{k'}.$$
	Therefore, by Lemma~\ref{lem:8}(3), we have
	$\v(k-k')=0.$
	
	(3) Suppose that it is false. Then there is $k'\in \calK$
	such that
	$\nS_{w_\star,k}\subsetneqq \nS_{w_\star,k'}\subseteq(d^\ur_{k'}, d^\Iw_{k}-d^\ur_{k'})$. By (2) and Lemma~\ref{lem:1}, this implies $L_{w_\star,k'}=0$, which clearly is impossible. 
	
	(4) By \cite[Lemma~5.15]{xiao}, it is enough to show that $$g_{d^\ur_k,\hat k}(w_k)=g_{d^\ur_k,\hat k}(w_\star) \quad\textrm{and}\quad g_{d^\Iw_k-d^\ur_k,\hat k}(w_k)=g_{d^\Iw_k-d^\ur_k,\hat k}(w_\star).$$
	By symmetry, we will just prove the first one. For $k'\in \calK\backslash\{k\}$ such that $m_{d^\Iw_k}(k')\geq 1$ and
	$\v(w_k-w_{k'})<\v(w_k-w_{\star})$, we have 	$$\v(w_k-w_{k'})=\v(w_{\star}-w_{k'}).$$
	Hence, we just need to study $k'\in \calK\backslash\{k\}$ such that $m_{d^\Iw_k}(k')\geq 1$ and
	$\v(w_k-w_{k'})\geq \v(w_k-w_{\star})$.  By \cite[Proposition~5.16(1)]{xiao}, we have 
	 $\nS_{w_\star,k}\subseteq (d^\ur_{k'}, d^\Iw_{k'}-d^\ur_{k'})$, and hence by (2), 	$\v(w_k-w_{k'})=\v(k-k')+1=1$.
	 This is a contradiction to the following chain of inequality from \cite[Lemmas~5.2 and 5.6]{xiao} 
	\[	\v(w_k-w_{k'})\geq \v(w_k-w_{\star})\geq \OD_{k,\frac12d_{k}^\new}-\OD_{k,\frac12d_{k}^\new-1}\geq \frac{3}{2}.\qedhere\]
\end{proof}

We now prove our main theorem assuming the following proposition, whose proof is given in the next section.
\begin{proposition}\label{prop1}
	Let $k_0\in \calS$ such that 
	\begin{equation}\label{eqt:2}
		n_0:=d_{k_0}^\ur \geq d_{k}^\Iw-d_{k}^\ur \textrm{~for any~}  k\in \calS \textrm{~with~} k<k_0.
	\end{equation} For any $\tk\in \calS$, if $\NP(G(w_{\tk},-))$ contains a line segment $\overline{PQ}$ with endpoints $n_0$ and $n_0+L$ for some $L\geq 1$, then $\overline{PQ}$ has slope 	$\geq \min\{m-4, \frac{k_0-2}{2}\}$. 
\end{proposition}
\begin{proof}[Proof of Theorem~\ref{thm1}]
	Choose a weight $\tk\in\calS$ arbitrarily.  
	It is enough to prove that there exists some integer $n_0$ such that 
	\begin{enumerate}[label=(\alph*)]
		\item the Newton polygon of $G(w_{\tk},-)$ has vertex at $i=n_0$.
		\item The Newton polygon of $G(w_{\tk},-)$ agrees with  $G(w_{\ko},-)$ up to the index $i=n_0$.
		\item The line segment in $\NP(G(w_{\tk},-))$ with left endpoint $n_0$ has slope $\geq m-4$.
	\end{enumerate}
	Note first that from the assumption on $\tk$ and construction of $k_0^{\min }$, we have \begin{equation}\label{eqr6}
		v_p(k-\tk)=v_p(k-k_0^{\min }) \textrm{~for any~} k\in \calK\backslash\{\ko\} \textrm{~with~} k<\kO. 
	\end{equation}

	We divide our discussion into two scenarios.
	
	\noindent \textbf{Case 1:} $\frac{\ko-2}{2}\geq m-4$. 
	
	In this case, we take $n_0:=d_{k_0^{\min }}^{\mathrm{ur}}$.
	Now we prove (a). Suppose  that $n_0=d_{k_0^{\min }}^{\mathrm{ur}}$ is not a vertex by contradiction, i.e. the index $i=n_0$ would be part of a segment $(n_1, n_2)$ with $n_1<n_0<n_2$ over which $\NP\left(G\left(w_{\tk}, -\right)\right)$ is a straight line.  By Proposition~\ref{xiao:pro}(2), $(n_1, n_2)$ is a maximal near-Steinberg range, say $(n_1, n_2)=\nS_{w_{\tk}, k'}$ for some $k'\in \calK$. By definition, we have $d_{k'}^{\mathrm{ur}}\leq n_1<d_{k_0^{\min }}^{\mathrm{ur}}$, and hence $k'<k_0^{\min }$. Combined with \eqref{eqr6}, this strict inequality implies that
	$$\v(k'-k_0^\omin)<m\leq \v(\tk-k_0^\omin),$$
	a contradiction to Proposition~\ref{xiao:pro}(1). This proves (a).
	
Combining (a) with \eqref{eqr6}, we complete the proof of (b). Since $k_0^{\min }$ is the minimal weight in $\calS$, the condition \eqref{eqt:2} in Proposition~\ref{prop1} is trivial. Hence, combining  Proposition~\ref{prop1} and our assumption $\frac{\ko-2}{2}\geq m-4$ in this case, we complete the proof of (c). 
	
	\noindent \textbf{Case 2:} $\frac{\ko-2}{2}< m-4$. 
	In this case, we take $n_0:=d_{\kO}^{\ur}$, and first prove	\begin{equation}\label{eqt:3}
		d_{\ko}^{\mathrm{ur}}+d_{\ko}^{\new}\leq d_{\kO}^{\mathrm{ur}},
\end{equation} that is the condition \eqref{eqt:2} for $k_0^{\max}$ in Proposition~\ref{prop1}.
If $d_{\ko}^{\new}=0$, then by Lemma~\ref{lem:8}, we have 	$d_{\ko}^{\ur}\leq d_{\kO}^{\mathrm{ur}}.$
This proves \eqref{eqt:3} for this case. Now we assume that $d_{\ko}^{\new}\geq 1$. 
	By Proposition~\ref{lemma2}, we have
	$$ \OD_{\ko,\frac12d_{\ko}^\new}-\OD_{\ko,\frac12d_{\ko}^\new-1}\leq \frac{\ko-2}{2}+4< m.$$
Combining this inequality with $$\v(w_{\tk}-w_{\ko})=\v(\tk-\ko)+1\geq m+1$$ and Lemma~\ref{lem:2}, we obtain that $\NP\left(G\left(w_{\tk}, -\right)\right)$ contains a line segment with endpoints $d_{\ko}^{\ur}, d_{\ko}^{\mathrm{ur}}+d_{\ko}^{\new}$ and of slope $\frac{\ko-2}{2}$. On the other hand,  in this case,  we have
	\begin{equation*}
		d_{\ko}^{\mathrm{ur}}+d_{\ko}^{\new}\leq  d_{\ko}^{\Iw}\leq 2\kob+2\leq \frac{2\ko}{p-1}+2.
	\end{equation*}
	Recall that by	Notation~\ref{notation:2}(2), we have $k^{\max}_{0\bullet}\geq p^m.$
Applying  the condition $\frac{\ko-2}{2}< m-4$, our assumption $m\geq 4$ and \eqref{eqw1} consecutively, we have
	$$\frac{2\ko}{p-1}+2\leq  \frac{4m}{p-1}+2\leq \frac{2p^{m}}{p+1}-2\leq  \frac{2k^{\max}_{0\bullet}}{p+1}-2\leq d_{\kO}^{\mathrm{ur}}.$$
	This completes the proof of \eqref{eqt:3}.

	By \eqref{eqr6} and a similar argument to Case 1, we complete the proof of (a)(b) for this case. 
	Combining 
	$\frac{\kO-2}{2}\geq p^m\geq m$ and \eqref{eqt:3} with Proposition~\ref{prop1}, we complete the proof of (c) for this case. 
\end{proof}

%
\section{Proof of Proposition~\ref{prop1}}\label{s4}

Throughout this section, let $n_0$ and $L$ be defined as in Proposition~\ref{prop1}, and recall that $p\geq 11$.
The proof of Proposition~\ref{prop1} is divided into two cases $L=1$ and $L\geq 2$, and each case is further broken into sub-cases. 
We first introduce a tool lemma for ease of use later (since the proof is not difficult).
\begin{lemma}\label{lem:6}
	For any integers $n_1<n_2$,
choose $r\in \RR_{\geq 0}$ such that $\v(n)\leq r$ for every $n_1+1\leq n\leq n_2$, then  $$\frac{n_2-n_1}{p}-1\leq \sum_{n=n_1+1}^{n_2} \v(n)\leq \frac{n_2-n_1}{p-1}+r.$$
\end{lemma}

\begin{proof}
	The hypothesis in this lemma implies that
	$$	\left\lfloor\frac{n_2}{p^i}\right\rfloor- \left\lfloor\frac{n_1}{p^i}\right\rfloor=0	\quad\textrm{for every~} i\geq  \lfloor r\rfloor+1.$$
	Hence,
	$$\sum_{n=n_1+1}^{n_2} \v(n) =\sum_{i=1}^\infty\left\lfloor\frac{n_2}{p^i}\right\rfloor-\sum_{i=1}^\infty\left\lfloor\frac{n_1}{p^i}\right\rfloor\leq \sum_{i=1}^{\lfloor r\rfloor}\left\lfloor\frac{n_2-n_1}{p^i}\right\rfloor+\lfloor r\rfloor
	\leq \frac{n_2-n_1}{p-1}+r,$$
	which proves the second inequality.
	The first inequality follows from  
	\[	\left\lfloor\frac{n_2}{p}\right\rfloor- \left\lfloor\frac{n_1}{p}\right\rfloor\geq 	\left\lfloor\frac{n_2-n_1}{p}\right\rfloor\geq \frac{n_2-n_1}{p}-1.\qedhere\]
\end{proof}

\begin{proof}[Proof of Proposition~\ref{prop1} for the case $L=1$]
	Recall that $n_0:=d_{k_0}^\ur$. 
	Note that in this case, the slope of $\overline{PQ}$
is equal to 	$\v(g_{n_0+1}(w_{\tk}))-\v(g_{n_0}(w_{\tk}))\in \ZZ$.
It is enough to prove 
	$$\v(g_{n_0+1}(w_{\tk}))-\v(g_{n_0}(w_{\tk}))> m-5.$$
Clearly, 
%
	\begin{equation}\label{eq:27}
		\begin{split}
	&	\v(g_{n_0+1}(w_{\tk}))-\v(g_{n_0}(w_{\tk}))
				\\=	&	\left(\v(k_0-\tk)+1\right)+\sum_{
				\substack{k\in \calK\backslash\{k_0\}\\d^\ur_k\leq n_0 < \frac{1}{2}d^\Iw_k}}
			\left(\v(k-\tk)+1\right)-\sum_{
				\substack{k\in \calK\\\frac{1}{2}d^\Iw_k\leq n_0 \leq d^\Iw_k-d^\ur_k -1}}
			\left(\v(k-\tk)+1\right)
		\end{split}
	\end{equation}
	We first prove that  any weight $k\neq k_0$ that is showed in the two sums of \eqref{eq:27} satisfies
 \begin{equation}\label{eq:22}
	\v(k-\tk)< m.
\end{equation}
Let $k$ be any such weight. We note that $d^\ur_k\leq n_0.$
If $k_\bullet>k_{0\bullet},$ then $d^\ur_k= n_0$, and by Lemma~\ref{lem:8}(3), 
 $\v(k-k_0)=0< m$.
If $k_\bullet<k_{0\bullet},$ from our assumption on $k_0$ in this proposition, we have $k\notin \calS$, and hence
$\v(k-k_0)< m$. Combined with $\v(\tk-k_0)\geq m$, these two inequalities above complete the proof of 
\eqref{eq:22}. 

	We now estimate the second sum, and take $k$ arbitrarily from this sum, which clearly is strictly less than $k_0$. Applying Lemma~\ref{lem:8}(1) on	$$\frac{1}{2}d^\Iw_k\leq n_0 \leq d^\Iw_k-d^\ur_k -1\leq d^\Iw_k-1,$$ we have
	\begin{equation}\label{eq:23}
	\left\lceil\frac{n_0-1}{2}\right\rceil\leq k_\bullet\leq n_0.
	\end{equation}
	Combined with \eqref{eq:22}, $\v(k_{0}-\tk)\geq m$ and $\v(k'-k'')=\v(k'_\bullet-k''_\bullet)$ for every $k',k''\in \calK$,
	we have
	\begin{equation}\label{eq6}
		\v(k_{\bullet}-\tk_{\bullet})=\v(k_{\bullet}-k_{0\bullet})\leq 
	\log_p \left(k_{0\bullet}-k_{\bullet}\right)\leq 
	\log_p \left(k_{0\bullet}-\frac{n_0-1}{2}\right).
	\end{equation}
By \eqref{eqw1} again, we have
	$k_{0\bullet}\leq (p+1)(\frac{n_0}{2}+1)$. 
	Plugging  this inequality into \eqref{eq6}, we have
	\begin{equation*}
		\v(k_{\bullet}-k_{0\bullet})\leq \log_p\left(\frac{n_0p}{2}+p+\frac{3}{2}\right).
	\end{equation*}
	Combined with \eqref{eq:23} and  Lemma~\ref{lem:6} with $r=\log_p\left(\frac{n_0p}{2}+p+\tfrac32\right)$,  this inequality shows that
	\begin{multline}\label{eq:7}
		  \sum_{
			\substack{k\in \calK\\\frac{1}{2}d^\Iw_k\leq n_0 \leq d^\Iw_k-d^\ur_k -1}}
		\left(\v(k-\tk)+1\right)\\
		\begin{aligned}
			 	\leq &
			\sum_{\lceil\frac{n_0-1}{2}\rceil\leq k_\bullet\leq n_0}
			\left(1+\v(k_{\bullet}-k_{0\bullet})\right)
			\\\leq& 
			n_0-\left\lceil\frac{n_0-1}{2}\right\rceil+1+\frac{n_0-\lceil\frac{n_0-1}{2}\rceil+1}{p-1}+\log_p\left(\tfrac{n_0p}{2}+p+\tfrac{3}{2}\right)\\
		\leq  &\frac{p(n_0+3)}{2(p-1)} +\log_p\left(\tfrac{n_0p}{2}+p+\tfrac{3}{2}\right).
		\end{aligned}
	\end{multline}

	We divide our discussion into two scenarios.
	
		\noindent \textbf{Case 1:}	$0\leq n_0\leq 3$. By \eqref{eq:7}, we have
	\begin{align*}
		&\v(g_{n_0+1}(w_{\tk}))-\v(g_{n_0}(w_{\tk}))\\
		\geq &
		\left(\v(k_0-\tk)+1\right)-\sum_{
			\substack{k\in \calK\\
				\frac{1}{2}d^\Iw_k\leq n_0 \leq d^\Iw_k-d^\ur_k -1}}
		\left(\v(k-\tk)+1\right)
		\\
		\geq &	\left(\v(k_0-\tk)+1\right)-\left(\frac{p(n_0+3)}{2(p-1)} +\log_p\left(\tfrac{n_0p}{2}+p+\tfrac32\right)\right)
		\\	\geq& m-\tfrac{3p}{p-1}-1	
		\\	\geq& m-4.3,	\end{align*}
	where the last two inequalities are from our assumption $p\geq 11$.
	
	\noindent \textbf{Case 2:} $n_0\geq 4$. Note that if $n_0< k_\bullet< (p+1)\left(\frac{n_0}{2}-1\right),$ then by Lemma~\ref{lem:8}(1), we have
	$$n_0<k_\bullet\leq \frac{1}{2}d_k^\Iw\quad\textrm{and}\quad d^\ur_k\leq\frac{2k_\bullet}{p+1}+2<	n_0,$$
	i.e. $d^\ur_k<n_0<\frac{1}{2}d_k^\Iw$.
	Combined with Lemma~\ref{lem:6} and the fact that $p\geq 11$ is odd, this implies
	\begin{multline}\label{eq:4}
		\sum_{
			\substack{k\in \calK\\d^\ur_k< n_0 < \frac{1}{2}d^\Iw_k}}
		\left(\v(k-\tk)+1\right)\\
		\begin{aligned}
			\geq &	\sum_{n_0< k_\bullet< (p+1)\left(\tfrac{n_0}{2}-1\right)}
			\left(\v(k_{\bullet}-\tk_{\bullet})+1\right)
			\\
			\geq&
			\frac{(p+1)\left(\frac{n_0}{2}-1\right)-n_0-1}{p}-1+
				(p+1)\left(\tfrac{n_0}{2}-1\right)-n_0-1\\
			=&
			(p+1)	\frac{(p+1)\left(\frac{n_0}{2}-1\right)-n_0-1}{p}-1\\
			=&\frac{(p^2-1)n_0}{2p}-\frac{(p+1)(p+2)}{p}-1.
		\end{aligned}
	\end{multline}
		Plugging \eqref{eq:4} and \eqref{eq:7} consecutively into \eqref{eq:27}, we obtain
	\begin{multline*}
		\v(g_{n_0+1}(w_{\tk}))-\v(g_{n_0}(w_{\tk}))
		\geq \v(k_{0}-\tk)+\frac{(p^2-1)n_0}{2p}-\frac{(p+1)(p+2)}{p}
		\\-\left(\frac{p(n_0+3)}{2(p-1)}+\log_p\left(\frac{n_0p}{2}+p+\tfrac32\right)\right).
	\end{multline*}
	From our assumption $p\geq 11$ and $n_0\geq 4$, the term 
	$$	\frac{(p^2-1)n_0}{2p}-\left(\frac{p(n_0+3)}{2(p-1)}+\log_p\left(\frac{n_0p}{2}+p+\tfrac32\right)\right)$$
	reaches its minimum at $n_0=4$, and hence
	\begin{align*}
		&	 \frac{(p^2-1)n_0}{2p}-\frac{(p+1)(p+2)}{p}
		-\left(\frac{p(n_0+3)}{2(p-1)}+\log_p\left(\frac{n_0p}{2}+p+\tfrac32\right)\right)\\
		\geq& 	 \frac{2(p^2-1)}{p}-\frac{(p+1)(p+2)}{p}
		-\left(\frac{7p}{2(p-1)}+\log_p\left(3p+\tfrac32\right)\right)
		\\	=&\frac{p^2-3p-4}{p}-\frac{7p}{2(p-1)}-\log_p(3p+\tfrac{3}{2})\\
		= &
		\frac{2p^3-15p^2-2p+8}{2p(p-1)}-\log_p(3p+\tfrac{3}{2}).
	\end{align*}
	Note that the last term in this chain of inequality reaches its minimum at $p=11$ with value $> 2.3.$
	Hence, we have $\v(g_{n_0+1}(w_{\tk}))-\v(g_{n_0}(w_{\tk}))\geq m-4,$
	which completes the proof of this case.
\end{proof}

\begin{proof}[Proof of Proposition~\ref{prop1} for the case $L\geq 2$]
	By Proposition~\ref{xiao:pro}(2), 
	there exists $k'\in \calK$ such that
	\begin{enumerate}
		\item $\nS_{w_{\tk},k'}$ is a maximal near-Steinberg range.
		\item $\nS_{w_{\tk},k'}=(n_0, n_0+L)$. 
		\item $L_{w_{\tk}, k'}=\frac{L}{2}\geq 1$.
			\item The slope of $\overline{PQ}$
			is equal to 
			$ \frac{1}{L}\left(\v\left(g_{\frac{1}{2}(d_{k'}^{\Iw}+L),\widehat{k'}}(w_{\tk})\right)
			 	-\v\left(g_{\frac{1}{2}(d_{k'}^{\Iw}-L),\widehat{k'}}(w_{\tk})\right)\right).$
	\end{enumerate}  
Hence, to prove this proposition, it is enough to show that 
	$$ \frac{1}{L}\left(\v\left(g_{\frac{1}{2}(d_{k'}^{\Iw}+L),\widehat{k'}}(w_{\tk})\right)
-\v\left(g_{\frac{1}{2}(d_{k'}^{\Iw}-L),\widehat{k'}}(w_{\tk})\right)\right)\geq \min\{m-4,\tfrac{k_0-2}{2}\}.$$

\noindent \textbf{Case 1:} ${k'}=k_0$. Note that in this case, we have $\nS_{w_{\tk},k_0}= (d^\ur_{k_0}, d^\Iw_{k_0}-d^\ur_{k_0})$, 
 and hence
	$$\v(w_{k_0}-w_{\tk})\geq   \OD_{k_0,\frac12d_{k_0}^\new}-\OD_{k_0,\frac12d_{k_0}^\new-1}.$$
By Lemma~\ref{lem:2}(4), we complete the proof of this case.

 \noindent \textbf{Case 2:} $k'\neq k_0$.  
 Our strategy is to studying the difference $\bT$ between our goal term $$\dfrac{1}{L}\left(\v\left(g_{\frac{1}{2}(d_{k'}^{\Iw}+L),\widehat{k'}}(w_{\tk})\right)
 -\v\left(g_{\frac{1}{2}(d_{k'}^{\Iw}-L),\widehat{k'}}(w_{\tk})\right)\right)$$
 and
  \begin{equation*}
 	\frac{1}{L}\left(\v\left(g_{\frac{1}{2}(d_{k'}^{\Iw}+L),\widehat{k'}}(w_{k'})\right)
 	-\v\left(g_{\frac{1}{2}(d_{k'}^{\Iw}-L),\widehat{k'}}(w_{k'})\right)\right)=\frac{k'-2}{2},
 \end{equation*}
 where the above inequality is from \eqref{neq::19}. This benefits from the following reduction: for any $n\geq 0$, $$\v \left( g_{n,\widehat{k'}}(w_{\tk})\right)-\v \left(g_{n,\widehat{k'}}(w_{k'})\right)=\sum_{k\in \calK} m_n^{(\varepsilon)}(k) (\v(w_k-w_{\tk})-\v(w_k-w_{k'})).$$
Note that for $k\in \calK$ such that $\v(w_k-w_{k'})<\v(w_{\tk}-w_{k'})$, we have 
$$ \v(w_k-w_{k'})=\v(w_k-w_{\tk}).$$
Hence, $$\v \left( g_{n,\widehat{k'}}(w_{\tk})\right)-\v \left(g_{n,\widehat{k'}}(w_{k'})\right)=\sum_{\substack{k\in \calK\backslash\{k'\}\textrm{~s.t.~}\\ \v(w_k-w_{k'})\geq \v(w_{\tk}-w_{k'})}} m_n^{(\varepsilon)}(k) (\v(w_k-w_{\tk})-\v(w_k-w_{k'})).$$
Let $$\calK':=\left\{k\in \calK\backslash\{k'\}\;\Big|\; \v(w_k-w_{k'})\geq \v(w_{\tk}-w_{k'});  m_{\frac{1}{2}(d_{k'}^{\Iw}-L)}^{(\varepsilon)}(k)> 0 \textrm{~or~}m_{\frac{1}{2}(d_{k'}^{\Iw}-L)}^{(\varepsilon)}(k)> 0\right\}.$$
 Hence, $$\bT=\sum_{k\in \calK'} \left(m_{\frac{1}{2}(d_{k'}^{\Iw}+L)}^{(\varepsilon)}(k)-m_{\frac{1}{2}(d_{k'}^{\Iw}-L)}^{(\varepsilon)}(k)\right) (\v(w_k-w_{\tk})-\v(w_k-w_{k'})).$$
 Note that for $k\in \calK\backslash\{k'\}$ such that $m_{\frac{1}{2}(d_{k'}^{\Iw}-L)}^{(\varepsilon)}(k)> 0 \textrm{~or~}m_{\frac{1}{2}(d_{k'}^{\Iw}-L)}^{(\varepsilon)}(k)> 0$, we have
 $(d_{k}^\ur, d_{k}^\Iw-d_{k}^\ur)\cap \overline{\nS}_{w_{\tk},k'}\neq \emptyset$, where $\overline{\nS}_{w_{\tk},k'}$ is defined in Notation-Lemma~\ref{notation1}.
 Hence, by Proposition~\ref{xiao:pro}(1), each $k\in \calK'$ must satisfies either $$ \nS_{w_{\tk},k'} \subseteq (d_{k}^\ur, \tfrac{1}{2}d_{k}^\Iw)\quad \textrm{or}\quad \nS_{w_{\tk},k'}\subseteq (\tfrac{1}{2}d_{k}^\Iw, d_{k}^\Iw-d_{k}^\ur).$$ 
 This gives us $ \calK'=\calK^+ \sqcup\calK^-$, where $$\calK^+:=\{k\in \calK'\;|\; \v(w_{k}-w_{k'})\geq \v(w_{\tk}-w_{k'})\textrm{~and~}\nS_{w_{\tk},k'} \subseteq (d_{k}^\ur, \tfrac{1}{2}d_{k}^\Iw)\},$$ 
 $$\calK^-:=\{k\in \calK'\;|\; \v(w_{k}-w_{k'})\geq \v(w_{\tk}-w_{k'})\textrm{~and~} \nS_{w_{\tk},k'}\subseteq (\tfrac{1}{2}d_{k}^\Iw, d_{k}^\Iw-d_{k}^\ur)\}.$$
Note that 
$$m_{\frac{1}{2}(d_{k'}^{\Iw}+L)}^{(\varepsilon)}(k)-m_{\frac{1}{2}(d_{k'}^{\Iw}-L)}^{(\varepsilon)}(k)=\begin{cases}
	L&\textrm{for~}k\in \calK^+,\\
	-L&\textrm{for~}k\in \calK^-.
\end{cases}$$
Hence, we have
 \begin{equation*}
 	\begin{split}
 	\bT	=&\sum_{k\in \calK^+}(\v(w_{k}-w_{\tk})-\v(w_{k}-w_{k'}))-\sum_{k\in \calK^{-}}(\v(w_{k}-w_{\tk})-\v(w_{k}-w_{k'})),
 	\end{split}
 \end{equation*}
 that is equivalent to
  \begin{equation}\label{eq:17}
 	\begin{split}
 		&	\frac{1}{L}\left(\v\left(g_{\frac{1}{2}(d_{k'}^{\Iw}+L),\widehat{k'}}(w_{\tk})\right)
 		-\v\left(g_{\frac{1}{2}(d_{k'}^{\Iw}-L),\widehat{k'}}(w_{\tk})\right)\right)\\
 		=&
 		\frac{{k'}-2}{2}+\sum_{k\in \calK^+}(\v(w_{k}-w_{\tk})-\v(w_{k}-w_{k'}))-\sum_{k\in \calK^{-}}(\v(w_{k}-w_{\tk})-\v(w_{k}-w_{k'})).
 	\end{split}
 \end{equation}
We now give $	\frac{1}{L}\left(\v\left(g_{\frac{1}{2}(d_{k'}^{\Iw}+L),\widehat{k'}}(w_{\tk})\right)
-\v\left(g_{\frac{1}{2}(d_{k'}^{\Iw}-L),\widehat{k'}}(w_{\tk})\right)\right)$ a lower bound by estimating two summations above.
 
 We first prove that
\begin{equation}\label{eqt:5}
	k_0\in \calK^+. 
\end{equation}
Note that by \cite[Lemma~5.2]{xiao} and the definition of $\nS_{w_{\tk},k'}$, we have 
\begin{equation}\label{eqt:4}
	\v(k'-\tk)=\v(w_{k'}-w_{\tk})-1\geq 1,
\end{equation}
and hence $\v(k_0-{k'})>0$, since otherwise, we have
$\v(\tk-{k'})=0$. By Lemma~\ref{lem:1}, this implies that $L_{w_{\tk}, k'}=0$, a contradiction to the property (3) of $k'$.

We prove $k'<k_0$ as follows: since $n_0$ is a vertex of $\nS_{w_{\tk},k'}$, we must have 
$d^\ur_{k'}\leq n_0$. Hence, by Lemma~\ref{lem:8}(2), it is enough to show 
$d^\ur_{k'}\neq n_0.$ Suppose otherwise. By Lemma~\ref{lem:8}(3), we have
$\v(k_0-{k'})=0$, a contradiction to $\v(k_0-{k'})>0$. 

From $k'<k_0$ and $n_0< d_{k'}^\Iw-d_{k'}^\ur$, the given assumption on $k_0$ in this proposition implies
 $\v(k'-\tk)<m.$ As a consequence, we complete the proof of \eqref{eqt:5}. 

We now prove 	that	\begin{equation}\label{eqr12}
	k\leq k_0\textrm{~for every~}k\in \calK'.
\end{equation}
			 For every $k\in \calK'$, we have
		$\v(k-k')\geq \v(\tk-k')>0$, where the last inequality follows from \eqref{eqt:4}. Combined with  $\v(k_0-{k'})>0$, we have $\v(k-k_0)>0$. Note that $d_k^\ur\leq n_0$. By Lemma~\ref{lem:8}(3), if 
		$d_k^\ur= n_0$, then $\v(k-k_0)=0$, a contradiction. This completes the proof of \eqref{eqr12}.
	
			Combining \eqref{eqr12} with the assumption on $k_0$ in this proposition, we obtain that for every $ k\in \calK'\backslash\{k_0\}$, we have
	\begin{equation}\label{eq:9}
		\v(k-\tk)<m.
	\end{equation}
	Note that 
	\begin{equation}\label{eqt:6}
		v_p\left(w_{k_1}-w_{k_2}\right) \geq 1 \textrm{~for all~} k_1, k_2\in \calK.
	\end{equation}
	Combining this with \eqref{eqt:5}, we have
	\begin{equation}\label{eq:12}
		\sum_{k\in \calK^+}(\v(w_{k}-w_{\tk})-\v(w_{k}-w_{k'}))\geq \v(w_{k_0}-w_{\tk})-1+\sum_{k\in \calK^+}(1-\v(w_{k}-w_{k'})), 
	\end{equation}
and by \eqref{eqw1}, 
	\begin{equation}\label{eq:36}
		k_{0\bullet}\leq \left(\frac{n_0}{2}+1\right)(p+1).
	\end{equation} 
From $k'<k_0$ and \eqref{eqr12},  for every 
$k\in \calK'$ 
we have
	%
\begin{equation}\label{eq:29}
		\v(k_{\bullet}-k'_\bullet)\leq \log_p(|k_{\bullet}-k'_{\bullet}|)\leq \log_p(k_{0\bullet})\leq \log_p \left(\frac{n_0}{2}+1\right)+\log_p(p+1).
\end{equation}
If $k\in \calK^+$, by definition, we have
$
d_{k}^{\ur}	\leq  n_0\leq 	\frac{1}{2}	d_{k}^{\Iw}-1
$,
and hence by Lemma~\ref{lem:8}(1),
		$$ n_0\leq  k_\bullet\leq \left(\frac{n_0}{2}+1\right)(p+1).$$
	Combined with \eqref{eq:29} and Lemma~\ref{lem:6} with $r=\log_p \left(\frac{n_0}{2}+1\right)+\log_p(p+1)$, this restriction of $k_\bullet$ implies that
	\begin{align*}
		&	\sum_{k\in \calK^+}(1-\v(w_{k}-w_{k'}))\\
	\geq &-	\sum_{n_0\leq  k_\bullet\leq \left(\frac{n_0}{2}+1\right)(p+1)}\v(k_{\bullet}-k'_\bullet)
		\\
		\geq& 
		-\left(\frac{\left(\frac{n_0}{2}+1\right)(p+1)-n_0+1}{p-1}\right)-\log_p \left(\frac{n_0}{2}+1\right)-\log_p(p+1)\\
			\geq &-\frac{n_0}{2}-\frac{p+2}{p-1}-\log_p\left(\frac{n_0}{2}+1\right)-\log_p(p+1).
	\end{align*}

	Plugging it into \eqref{eq:12}, we have
	\begin{multline}\label{eq:13}
		\sum_{k\in \calK^+}(\v(w_{k}-w_{\tk})-\v(w_{k}-w_{k'}))\\
	\geq \v(w_{k_0}-w_{\tk})-\frac{n_0}{2}-\log_p\left(\frac{n_0}{2}+1\right)-\frac{2p+1}{p-1}-\log_p(p+1).
	\end{multline}
	On the other hand, by \eqref{eqt:6} again, we have
	\begin{align}\label{eq:16}
		\sum_{k\in \calK^-}(\v(w_{k}-w_{\tk})-\v(w_{k}-w_{k'}))
		\leq &\sum_{k\in \calK^-}(\v(w_{k}-w_{\tk})-1)\\
		=&\sum_{k\in \calK^-}\v(k-\tk).\notag 
	\end{align}
Note that 
by \eqref{eq:9},  we have
$\v(k-\tk)=\v(k-k_0)$, and hence
$$\sum_{k\in \calK^-}\v(k-\tk)=\sum_{k\in \calK^-}\v(k-k_0)=\sum_{k\in \calK^-}\v(k_\bullet-k_{0\bullet}).$$
Since every $k\in \calK^-$ satisfies $k<k_0$, we have
\begin{equation}\label{eq:35}
	\v(k-\tk)\leq \log_p (k_{0\bullet}-k_\bullet).
\end{equation}
On the other hand, for every $k\in \calK^-$,
by definition, we have
$$	\tfrac{1}{2}	d_{k}^{\Iw}+1\leq n_0\leq d_{k}^{\Iw}-d_{k}^{\ur}\leq d_{k}^{\Iw},
$$
and hence by Lemma~\ref{lem:8}(1) again,
$$ \frac{n_0}{2}-1\leq k_\bullet\leq n_0-1.$$
Combined with \eqref{eq:35}, 
	this restriction of $k_\bullet$ implies that \begin{equation}\label{eq:15}
		\sum_{k\in \calK^-}\v(k-\tk)\leq \sum_{\frac{n_0}{2}-1\leq k_\bullet\leq n_0}\v (k_{0\bullet}-k_\bullet).
	\end{equation}
	By \eqref{eq:36}, we estimate $\log_p (k_{0\bullet}-k'_\bullet)$ by
	$$\log_p (k_{0\bullet}-k'_\bullet)\leq \log_p\left(\left(\frac{n_0}{2}+1\right)(p+1)-\left(\frac{n_0}{2}-1\right)\right)=\log_p\left(\frac{pn_0}{2}+p+2\right).$$
	Combining it with Lemma~\ref{lem:6} with $r=\log_p\left(\frac{pn_0}{2}+p+2\right)$, we obtain
	$$\sum_{\frac{n_0}{2}-1\leq k_\bullet\leq n_0-1}\v (k_{0\bullet}-k_\bullet)\leq \frac{n_0-\frac{n_0}{2}+1}{p-1}+\log_p\left(\frac{pn_0}{2}+p+2\right).$$
	Plugging it into \eqref{eq:15} and then \eqref{eq:16}, we have
	$$	\sum_{k\in \calK^-}(\v(w_{k}-w_{\tk})-\v(w_{k}-w_{k'}))\leq \frac{n_0+2}{2(p-1)}+\log_p\left(\frac{pn_0}{2}+p+2\right).$$
	
	Plugging it with \eqref{eq:13} into \eqref{eq:17}, we prove
	\begin{multline}\label{eq:18}
		\frac{1}{L}\left(\v\left(g_{\frac{1}{2}(d_{k'}^{\Iw}+L),\widehat{k'}}(w_{\tk})\right)
	-\v\left(g_{\frac{1}{2}(d_{k'}^{\Iw}-L),\widehat{k'}}(w_{\tk})\right)\right)\\
	\begin{aligned}
		\geq 
	\frac{k'-2}{2}+\v(w_{k_0}-w_{\tk})-\frac{n_0}{2}-\log_p\left(\frac{n_0}{2}+1\right)-\frac{2p+1}{p-1}-\log_p(p+1)&
	\\-\left(\frac{n_0+2}{2(p-1)}+\log_p\left(\frac{pn_0}{2}+p+2\right)\right)&.
		\end{aligned}
	\end{multline}
	Note that $n_0\leq \frac12d^{\Iw}_{k'}$. By Lemma~\ref{lem:8}(1), we have $n_0-1 \leq k_{\bullet}^{\prime}=\frac{k^{\prime}-k_{\varepsilon}}{p-1}$, and hence
	$$
	k^{\prime}-2 \geq k_{\varepsilon}-2+(p-1)\left(n_0-1\right) \geq(p-1)\left(n_0-1\right).
	$$
	Combined with the trivial bound $k'\geq 2$, this inequality gives
	$$k'\geq \max\{(p-1)(n_0-1), 0\}+2.$$
	Plugging it into 
	\eqref{eq:18}, we obtain 
	\begin{multline*}
				\frac{1}{L}\left(\v\left(g_{\frac{1}{2}(d_{k'}^{\Iw}+L),\widehat{k'}}(w_{\tk})\right)
			-\v\left(g_{\frac{1}{2}(d_{k'}^{\Iw}-L),\widehat{k'}}(w_{\tk})\right)\right)\\
	\begin{aligned}
		\geq 
\frac{\max\{(p-1)(n_0-1), 0\}}{2}+m+1-\frac{n_0}{2}-\log_p\left(\frac{n_0}{2}+1\right)-\frac{2p+1}{p-1}&\\
-\log_p(p+1)-\left(\frac{n_0+2}{2(p-1)}+\log_p\left(\frac{pn_0}{2}+p+2\right)\right)&.
	\end{aligned}
	\end{multline*}
By calculation and our assumption $p\geq 11$, the term on the right hand side in the above equation reaches its minimum when $p=11$ and $n_0=1$, which is
\begin{align*}
&	m+\frac{1}{2}-\log_p(\tfrac{3}{2})-\frac{2p+1}{p-1}-\log_p(p+1)-\left(\frac{3}{2(p-1)}+\log_p\left(\frac{3p}{2}+2\right)\right)\\
=&m+\frac{1}{2}-\frac{4p+5}{2(p-1)}-\log_p(\tfrac{3}{2})-\log_p(p+1)-\log_p\left(\frac{3p}{2}+2\right)\\
\geq& m-4.37.
\end{align*}
Note that by \eqref{eq:17}, we have	$$\frac{1}{L}\left(\v\left(g_{\frac{1}{2}(d_{k'}^{\Iw}+L),\widehat{k'}}(w_{\tk})\right)
-\v\left(g_{\frac{1}{2}(d_{k'}^{\Iw}-L),\widehat{k'}}(w_{\tk})\right)\right)\in \frac{\ZZ}{2}.$$  
This implies that $$\frac{1}{L}\left(\v\left(g_{\frac{1}{2}(d_{k'}^{\Iw}+L),\widehat{k'}}(w_{\tk})\right)
-\v\left(g_{\frac{1}{2}(d_{k'}^{\Iw}-L),\widehat{k'}}(w_{\tk})\right)\right)\geq m-4, $$ and hence
completes the proof.
\end{proof}


\begin{thebibliography}{9999}
		
	
		
		\bibitem[BC]{Buz}
		K. Buzzard and F. Calegari, A counterexample to the Gouvêa--Mazur conjecture, {\it C. R. Acad. Sci. Paris} {\bf 338}(2004), no. 10, 751--753.
		
			\bibitem[BP]{BP} J. Bergdall and R. Pollack, Slopes of modular forms and the ghost conjecture II, {\it Trans. Amer. Math. Soc.} {\bf 372}(2019), 357--388.
		
			\bibitem[CEGGPS]{ceggps} A. Caraiani, M. Emerton, T. Gee, D. Geraghty, V. Paskunas and S.-W. Shin, Patching and the $p$-adic local Langlands correspondence, {\it Camb. J. Math.} {\bf 4}(2016), 197--287.
		
		\bibitem[GM]{gm} F. Gouvêa and B. Mazur, Families of modular eigenforms, {\it Math. Comp.} {\bf 58}(1992), 793--805.
		
		\bibitem[Hat]{Hat} S. Hattori, Dimension variation of Gouvêa--Mazur type for Drinfeld cuspforms of level $\Gamma_1(t)$, {\it Int. Math. Res. Not.} {\bf 2021}(2021), no. 3, 2389--2402.
				

		
		\bibitem[LTXZ-1]{xiao}
		R. Liu, N. Truong, L. Xiao and B. Zhao, A local analogue of the ghost conjecture of Bergdall-Pollack, {\it Peking Math. J.} {\bf 7}(2024), 247--344.
		
		
				\bibitem[LTXZ-2]{xiao1}
		R. Liu, N. Truong, L. Xiao and B. Zhao,
		Slope of modular forms and geometry of eigencurves, {\it arXiv:2302.07697}.
		

		
		\bibitem[Wan]{wan} D. Wan, Dimension variation of classical and $p$-adic modular forms, {\it Invent. Math.} {\bf 133}(1998), 449--463.
	
	\end{thebibliography}
\end{document}